\documentclass[11pt,twoside]{amsart}
\usepackage{amsmath}
\usepackage{amsthm}
\usepackage{amsfonts}
\usepackage{amssymb}
\usepackage{dutchcal}
\usepackage{mathrsfs}
\usepackage{mathtools}
\usepackage{upgreek}
\DeclareMathAlphabet{\mathcal}{OMS}{cmsy}{m}{n}

\usepackage{enumerate}
\usepackage{indentfirst}
\usepackage{ulem}
\usepackage{tikz}
\usepackage{tikz-cd}
\usepackage{array}   
\newcolumntype{L}{>{$}l<{$}} 
\usepackage{verbatim}
\usepackage{hyperref} 
\usepackage{cleveref}
\usepackage{lipsum}
\usepackage{comment} 
\usepackage[toc]{appendix} 
\usepackage[colorinlistoftodos]{todonotes}

\usepackage[maxnames=99, style = alphabetic, doi=false,isbn=false,url=false]{biblatex}

\tikzset{
	symbol/.style={
		draw=none,
		every to/.append style={
			edge node={node [sloped, allow upside down, auto=false]{$#1$}}}
	}
}

\def \ord {\textup{ord}}
\def \mbc {\mathbb{C}}

\def \mbp {\mathbb{P}}

\def \mbp {\mathbb{P}}

\def \mbz {\mathbb{Z}}
\def \mbn {\mathbb{N}}

\def \ns  {\textup{NS}}

\def \ch {\textup{CH}}
\def \Bl {\textup{Bl}}
\def \ra {\rightarrow}

\def \id  {\textup{id}}
\def  \Hom {\textup{Hom}}
\def \lhom {\textup{hom}}

\renewbibmacro{in:}{} 
\DeclareFieldFormat{pages}{#1}

\newtheorem{thm}{Theorem}[section]
\newtheorem{lem}[thm]{Lemma}
\newtheorem{prop}[thm]{Proposition}
\newtheorem{cor}[thm]{Corollary}

\newtheorem*{conj-no}{Conjecture}

\theoremstyle{definition}
\newtheorem{defn}[thm]{Definition}
\newtheorem{rmk}[thm]{Remark}
\newtheorem{exmp}[thm]{Example}

\numberwithin{thm}{section}

\title{Elliptic constant cycle curves on Kummer surfaces}
\author{Jiexiang Huang }

\hypersetup{
	colorlinks = true,
	linkcolor=black
}

\begin{document}
	\begin{abstract}
	The order of a constant cycle curve $C \subset X$ on a K3 surface, defined by Huybrechts, is a positive integer that measures the obstruction to decomposing the diagonal class $\Delta_C$ in the Chow group $\ch^2(X \times C)$. 
	In this paper, we compute the order of  elliptic constant cycle curves that naturally arise on Kummer surfaces, by passing to the transcendental intermediate Jacobian $J_{\mathrm{tr}}^3(X \times C)$. As a consequence, every $n \in \mbn$ can be realized as the order of a constant cycle curve on a K3 surface.
	\end{abstract}
	\thanks{The author is supported by the ERC Synergy Grant HyperK (ID 854361)}
	\thanks{Address: Mathematical Institute, Universität Bonn, 53113 Bonn, Germany}
	\thanks{E-mail: jxhuang@math.uni-bonn.de}
	
	\maketitle

\section{Introduction} \label{intro}
 Constant cycle curves on surfaces were systematically studied by Huybrechts in \cite{Huybrechts2014}. These are curves whose every point represents the same class in the Chow group of the surface. 
 An integral curve $C \subset X$ on a K3 surface is a constant cycle curve if and only if there exists a nonzero integer $N$ such that $N[\Delta_C] \in \ch^2(X \times C)$ is decomposable, in the sense that it can be represented as a sum of product cycles, and the minimal positive integer $N$ with this property is defined by Huybrechts as the \textit{order} $\ord(C)$  of the constant cycle curve $C$. 

The order $\ord(C)$  of a constant cycle curve $C$ measures how far  $C$ is from being rational.  Rational curves are constant cycle curves of order one, but note that a constant cycle curve of order one is not necessarily rational. For a non-rational constant cycle curve $C$,  one typically  obtains an upper bound $\ord(C)\mid n$ by construction, but determining  $\ord(C)$ may be difficult,  see \Cref{sec_review} for a list of such examples. As for  rational curves, we have a finiteness result for constant cycle curves of fixed order $n$ in a linear system $|L|$ on $X$ (\cite[Prop.~5.1]{Huybrechts2014}). However,  there is still a lack of explicit examples of constant cycle curves of order larger than one. It is not even clear whether every positive integer can occur as the order of some constant cycle curve. The relationship between the order $\ord(C)$ and the genus $g(C)$ of a constant cycle curve also remains mysterious.

In this paper, we compute the order of a series of elliptic constant cycle curves on Kummer surfaces.
Recall that
the Kummer surface $X = \textup{Kum}(E_1 \times E_2)$, associated with the product of two elliptic curves, admits a natural elliptic fibration $p \colon X \ra E_1 \times E_2 / \pm \ra E_1 / \pm \simeq \mbp^1$. It has exactly four singular fibres, all of type $I_0^{\ast}$, corresponding to the two-torsion points of $E_1$. 
For any torsion point $t \in E_1$ of order $n$, the  fibre $E_t \coloneqq p^{-1}(\bar{t}) \simeq E_2$ is known to be a constant cycle curve with $\ord(E_t) \mid n$  by  \cite[Lem.~6.5]{Huybrechts2014}.  Note that for $n \leqslant 2$, the fibre $E_t$ is  singular and in particular rational, hence $\ord(E_t) = 1$. 

When $n > 2$, the constant cycle curve $E_t$ is smooth and elliptic. Under mild assumptions, we can determine $\ord(E_t)$. Our main result is as follows (see \Cref{order_kappa'_case1} and \Cref{ord_4torsion}):
\begin{thm} \label{ccc_Kummer_ord} 
	Let  $E_1, E_2$ be two elliptic curves and $t \in E_1$ be a torsion point of order $n > 2$. Then the constant cycle curve $E_t \subset X = \textup{Kum}(E_1 \times E_2)$ has
	\begin{equation*}
		\ord(E_t) = d(n) \coloneqq
		\begin{cases}
			n / 2,\quad & 2 \mid n \\
			n,\quad & 2 \nmid n
		\end{cases}
	\end{equation*}
except possibly when $E_1,E_2$ are non-isomorphic, isogeneous with CM and $4 \mid n$.
\end{thm}

In the excluded case, i.e.\ when $E_1 \sim E_2$  are non-isomorphic with CM and $4 \mid n$, determining $\ord(E_t)$ is subtle: it can be $n/2$ or strictly smaller,  see \Cref{counterexample} for an explicit example.

The proof of \Cref{ccc_Kummer_ord} follows a general principle suggested in \cite[Sec.~5.2]{Huybrechts2014}: The order of a constant cycle curve $C \subset X$ can be determined in the transcendental  intermediate Jacobian $J_{\mathrm{tr}}^3(X \times C)$; see  \Cref{sec_strategy} for the explicit formulation. 
The key ingredient of this strategy is the injectivity of Abel--Jacobi maps on codimension-two torsion classes, established by Colliot-Th{\'e}l{\`e}ne, Sansuc and  Soul{\'e} \cite{CTSS83AJcodimtwo}.

As a direct corollary of  \Cref{ccc_Kummer_ord}, the above elliptic constant cycle curves  provide  examples of constant cycle curves of arbitrary order $n > 1$, namely
\begin{cor} \label{ccc_everyorder}
	Let $n > 1$ be an integer. Any elliptic curve $E$ can be embedded as a constant cycle curve of exactly order $n$ into any  Kummer surface $X = \textup{Kum}(E \times F)$, where $F$ can be taken as $E$ or any elliptic curve non-isogeneous to $E$. 
\end{cor}

Recall from \cite[Ex.~7.3]{Huybrechts2014} that there exist smooth elliptic curves that are  constant cycle curves on K3 surfaces of order one. Therefore, we conclude
\begin{cor}
	For any integer $n \geqslant 1$, there exists a K3 surface $X$  with a smooth  constant cycle curve $E \subset X$ of genus one and  order $n$. 
\end{cor}

\noindent\textbf{Structure of this paper.} 
  We review the notion of order of constant cycle curves in  \Cref{sec_review}, and then outline a general strategy for determining  the order $\ord(C)$ of a constant cycle curve $C \subset X$ in the transcendental intermediate Jacobian $J_{\mathrm{tr}}^3(X \times C)$ in \Cref{sec_strategy}. 
  
  The proof of our main result \Cref{ccc_Kummer_ord} takes two steps. First, in \Cref{sec_order_Zt}, we  sharpen the upper bound  $\ord(E_t) \mid n$ to  $\ord(E_t) \mid d(n)$ by passing to the abelian threefold $E_1 \times E_2 \times E_t$.  Then we show in  \Cref{sec_ord_kprime} that this bound is generically optimal, i.e.\   $\ord(E_t) = d(n)$  under mild assumptions on $E_1$ and $E_2$. The case $E_1 \nsim E_2$ is handled via a study of product cycles in $\ch^2(E_1 \times E_2 \times E_t)$ (see \Cref{sec_decomp_triple_elliptic} and \Cref{sec_case1}), while in the more subtle case $E_1 \sim E_2$, we  apply the strategy in \Cref{sec_strategy} to
  determine  $\ord(E_t)$ in  $J_{\mathrm{tr}}^3(E_1 \times E_2 \times E_t)$ (see \Cref{sec_case2}). 
  Finally, an explicit example constructed in \Cref{counterexample} shows that the assumption in \Cref{ccc_Kummer_ord} cannot be easily removed. 

We remark that \Cref{ccc_Kummer_ord} is the first nontrivial application of the strategy  in \Cref{sec_strategy} for computing the order of non-rational constant cycle curves via intermediate Jacobians. We expect  to apply this general approach to  determine the order of other constructions of constant cycle curves on K3 surfaces, such as the ramification curve of a double plane.

\hfill\\
\textbf{Acknowledgement.}
I would like to thank my advisor  Daniel Huybrechts for numerous inspiring discussions and  valuable comments that greatly improved this paper. 

\hfill\\
\noindent\textbf{Convention.} Throughout this paper, we work over the field $\mbc$. For a smooth variety $X$, we denote by $\ch^k(X)$  the Chow group of cycle classes of codimension $k$ with integer coeffients, and $\ch^k(X)_{\lhom} \subset \ch^k(X)$ denotes the subgroup of homologically trivial cycles.  For any abelian group $A$ and $n \in \mbn$, we denote by $A[n]$ the subgroup of $n$-torsion elements. 
All curves are integral if not specified, and a smooth elliptic curve $E$ is always equipped with an origin $e$. By the notation $E_1 \sim E_2$ (resp.\ $E_1 \nsim E_2$) we mean that the two elliptic curves $E_1,E_2$ are isogeneous (resp.\ non-isogeneous).

\section{The order of constant cycle curves} \label{sec_review_strategy}
We review the notion of order of constant cycle curves and explain a general strategy  to compute the order of constant cycle curves by passing to the intermediate Jacobians. In this section, $X$ is always a complex projective K3 surface, and without loss of generality, we only consider integral constant cycle curves.
\subsection{Review: Definitions and examples} \label{sec_review}
Recall that  a curve $C \subset X$ is a constant cycle curve if $[x] = [y] \in \ch^2(X)$ for any $x,y \in \ch^2(X)$. Note that every point on a constant cycle curve  represents the Beauville--Voisin class $c_X \in \ch^2(X)$, see \cite[Lem.~2.2]{voisin2015rational}. 

For an integral curve $C \subset X$, denote by $\Delta_C \subset X \in C$ the diagonal cycle, and let   $\kappa_C$ be the image of $[\Delta_C] - c_X \times [C]$ under the following base change to the function field $k(C)$:
\begin{equation} \label{eq_kappa_c}
	\ch^2(X \times C) \ra \ch^2(X \times k(C)), \ \ [\Delta_C] - c_X \times [C] \mapsto \kappa_C.
\end{equation}
Then  we have the following equivalent definition of constant cycle curves.
\begin{lem}[{\cite[Lem.~3.4]{Huybrechts2014}}] 
	\label{def_ccc_kappa}
	An integral curve $C \subset X$ is a constant cycle curve if and only if the class $\kappa_C \in \ch^2(X \times k(C))$ is torsion.
\end{lem}

 This  motivates the following definition of the order of constant cycle curves.

\begin{defn}[{\cite[Def.~3.5]{Huybrechts2014}}] \label{def_ord_kappa}
	The \textit{order} of an integeral constant cycle curve $C \subset X$,  denoted by $\ord(C)$, is defined as the order of the associated torsion class $\kappa_C \in \ch^2(X \times k(C))$.
\end{defn}

We summarize what we know about the order of some typical constructions of constant cycle curves on K3 surfaces below. Note that only in rare cases  the explicit order is known.
\begin{exmp}[{\cite[Sec.~6-7]{Huybrechts2014}}]  \label{order_factor} 
	(1) A rational curve $C \subset X $ is of  order one. This is because $\ch^2(X \times \mbp^1) \simeq \ch^2(X) \otimes [\mbp^1] \oplus \ch^1(X) \otimes \ch^1(\mbp^1)$. \\
	(2) Any curve $C \subset X$  contained in the fixed locus of a non-symplectic automorphism of $X$ of order $n$ is a constant cycle curve with $\ord(C) \mid n$. For example, the ramification locus $B$ of a double plane $X \ra \mbp^2$ branched over a sextic has $\ord(B) \mid 2$. In particular, any curve $C$ contained in the fixed loci of two non-symplectic automorphisms of $X$ of coprime order has $\ord(C) = 1$.
	\\ 
	(3) Let $X = \textup{Kum}(E_1 \times E_2)$ be a Kummer surface, and $p \colon X \ra E_1 \times E_2/\pm \ra E_1/\pm \simeq \mbp^1$ be the natural elliptic fibration. For any torsion point $ t \in E_1$ of order $n >2$,  the fibre $E_{t} \coloneqq p^{-1}(\bar{t}) \simeq E_2$ is a smooth elliptic constant cycle curve with  $\ord(E_t) \mid n$. Here $\bar{t}$ denotes the image of $t$ under $E_1 \ra E_1 /\pm \simeq \mbp^1$.
\end{exmp} 

	To determine the order  $\ord(C)$ of  a non-rational constant cycle curve $C \subset X$, it is often more practical to work in $\ch^2(X \times C)$   rather than $\ch^2(X \times k(C))$. For this purpose, we reformulate \Cref{def_ccc_kappa} and  \Cref{def_ord_kappa}  in terms of  the decomposability of the diagonal class $[\Delta_C] \in \ch^2(X \times C)$ as follows.
	\begin{lem}\label{def_ord_equiv}
		An integral curve $C \subset X$ is a constant cycle curve if and only if there exists an integer $N > 0$ such that $N[\Delta_C]$ is \textit{decomposable}, namely 
		\begin{equation*}
			N[\Delta_C] = \sum_i D_i \times \alpha_i +
			\beta \times [C] \in \ch^2(X \times C)
		\end{equation*}
		for some $D_i \in \ch^1(X), \alpha_i \in \ch^1(C)$ and $\beta \in \ch^2(X)$. 
		
		In this case, we have  $\beta = N \cdot c_X \in \ch^2(X)$. In addition, the order $\ord(C)$ of the constant cycle curve $C$ is the minimal positive integer $N$ such that  $N[\Delta_C]$ is decomposable, or equivalently, such that
		$N([\Delta_C] - c_X \times [C])$ lies in the image of $\ch^1(X) \otimes \ch^1(C) \ra \ch^2(X \times C)$.
	\end{lem}

The above lemma follows immediately from  the fact below: 
\begin{lem}[{\cite[Lem.~1A.1]{Bloch_2010}}]\label{lem_bloch}
	Let $S$ be a smooth surface over an algebraically closed field and $C \subset S$ be an irreducible smooth curve. Then 
		$\ch^2(S \times k(C)) = \varinjlim_{U \subset C} \ch^2(S \times U)$,
	where $U$ goes over all nonempty open subsets of $C$. 
	In particular, the following sequence is exact:
	\begin{align*}
		\ch^1(S) \otimes \ch^1(C) &\ra  \ch^2(S \times C) \xrightarrow{r} \ch^2(S  \times k(C)) \ra 0\\
		\alpha \otimes \beta &\mapsto \big(p_{S}^{\ast}\alpha . p_{C}^{\ast}\beta \big). 
	\end{align*}
	\end{lem}

\subsection{Strategy to compute the order} \label{sec_strategy}
Determining the order of a non-rational constant cycle curve $C \subset X$ by working in the Chow group is usually a challenging task. In this section, we follow \cite[Sec.~5.2]{Huybrechts2014} and review the basic idea of computing $\ord(C)$ by passing to the transcendental intermediate Jacobian $J_{\mathrm{tr}}^3(X \times C)$. 

For a K3 surface $X$ and any smooth integral curve $C$, the short exact sequence of lattices
$$
0 \ra \ns(X) \ra H^2(X,\mbz) \ra T^{\prime}(X) \coloneqq \frac{H^2(X,\mbz)}{\ns(X)} \ra 0
$$
induces an extension of the intermediate Jacobian $J^3(X \times C)$:
\begin{equation}\label{ext_IJ}
	0 \ra J_{\mathrm{alg}}^{3}(X \times C) \ra J^3(X \times C) \ra J_{\mathrm{tr}}^{3}(X \times C) \ra 0,
\end{equation}
where 
\begin{align} \label{def_IJ_tr}
	J_{\mathrm{alg}}^3(X \times C) \coloneqq \frac{F^2(\ns(X) \otimes H^1(C))^{\ast}}{\ns(X) \otimes H^1(C,\mbz)} \simeq J(C)^{\oplus \rho(X)}, \ \  
	J_{\mathrm{tr}}^3(X \times C) \coloneqq  \frac{F^2(T^{\prime}(X) \otimes H^1(C))^{\ast}}{T^{\prime}(X) \otimes H^1(C,\mbz)}
\end{align} are the algebraic part and transcendental part of  $J^3(X \times C)$ respectively, and $\rho(X)$ is the Picard rank of $X$. 
By definition, the Abel--Jacobi map \begin{equation}\label{eq_AJmap}
	\Phi_X \colon \ch^2(X \times C)_{\lhom} \ra J^3(X \times C) = \frac{F^2(H^2(X) \otimes H^1(C))^{\ast}}{H_2(X,\mbz) \otimes H_1(C,\mbz)}, \ \ [Z] \mapsto \int_{\Gamma} \ \textup{with} \ \ \partial\Gamma = Z
\end{equation}
maps the image of $\ch^1(X) \otimes \ch^1(C)_{\lhom} \ra \ch^2(X \times C)_{\lhom}$ to $J_{\mathrm{alg}}^3(X \times C)$, hence we have the following commutative diagram with every row exact:
\begin{equation}\label{diag_algtrIJ}
	\begin{tikzcd}[sep=1.8em, font=\small]
	& &\ch^1(X) \otimes \ch^1(C)_{\lhom} \arrow[r] \arrow[d,"\Phi_{X}^{\mathrm{alg}}"] &\ch^2(X \times C)_{\lhom} \arrow[r] \arrow[d,"\Phi_{X}"] &\frac{\ch^2(X \times C)_{\lhom}}{\ch^1(X) \otimes \ch^1(C)_{\lhom}} \arrow[r] \arrow[d,"\Phi_X^{\mathrm{tr}}"] &0\\
	&0 \arrow[r] &J_{\mathrm{alg}}^3(X \times C) \arrow[r] &J^3(X \times C) \arrow[r, "p_{\mathrm{tr}}"] &J_{\mathrm{tr}}^3(X \times C) \arrow[r] & 0
\end{tikzcd},
\end{equation}
where $\Phi_{X}^{\mathrm{alg}} = \id_{\ch^1(X)} \otimes \Phi_C$ with $\Phi_C \colon  \ch^1(C)_{\lhom} \ra J(C)$ the Abel--Jacobi map, and $\Phi_X^{\mathrm{tr}}$ is the induced map. Note that the first row is the restriction of the short exact sequence in \Cref{lem_bloch} to the homologically trivial subgroups and one indeed has $\frac{\ch^2(X \times C)_{\lhom}}{\ch^1(X) \otimes \ch^1(C)_{\lhom}}  \simeq 
\ch^2(X \times k(C))_{0}$, the degree zero subgroup of $\ch^2(X \times k(C))$. We call  $$\Phi_X^{\mathrm{tr}} \colon \ch^2(X \times k(C))_{0} \ra J_{\mathrm{tr}}^3(X \times C)$$ the transcendental Abel--Jacobi map (note that this definition is non-standard).

Now, recall from \Cref{def_ord_kappa} that the order $\ord(C)$ of an integral constant cycle curve $C \subset X$  is defined as the order of the associated torsion class $\kappa_C \coloneqq (\Delta_C-c_X \times C)|_{X \times k(C)}$ in $\ch^2(X \times k(C))_{0}$. The following observation, using a result of Colliot-Th\'elène, Sansuc and Soul{\'e} \cite{CTSS83AJcodimtwo}, enables us to  compute  $\ord(C)$ in  $J_{\mathrm{tr}}^3(X \times C)$:
\begin{lem} \label{inj_tr_AJ}
	Let $X$ be a K3 surface and $C$ be a smooth integral curve. Then the transcendental Abel--Jacobi map $\Phi_X^{\mathrm{tr}}  \colon \ch^2(X \times k(C))_{0} \ra J_{\mathrm{tr}}^3(X \times C)$ is injective on the  torsion subgroup. 

\end{lem}
\begin{proof} 
	Let $\bar{\alpha} \in \ker(\Phi_X^{\mathrm{tr}})$ be a torsion class. Since  $\ch^1(X) \otimes \ch^{1}(C)_{\lhom}$ is divisible, we can  lift $\bar{a}$ to a  torsion class $\alpha \in \ch^2(X \times C)_{\lhom}$. 
	As  rows in \eqref{diag_algtrIJ} are exact and  $\Phi_X^{\mathrm{alg}}$ is an isomorphism, there exists  $\beta \in \ch^1(X) \otimes \ch^1(C)_{\lhom}$ such that $\Phi_X^{\mathrm{alg}}(\beta) = \Phi_X(\alpha)$. Since $\alpha$ is torsion by construction, so is $\beta$, and note that the torsion class $\beta-\alpha$ is contained in $\ker(\Phi_X)$. Then the injectivity of $\Phi_X$ on codimension-two torsion classes (see \cite{CTSS83AJcodimtwo}) forces $\beta = \alpha$, hence $\bar{\alpha} = 0$, i.e. there is no nonzero  torsion classes in $\ker(\Phi_X^{\mathrm{tr}})$. 
\end{proof}

The observation below follows directly from \Cref{inj_tr_AJ} and provides a practical method to compute $\ord(C)$ in the transcendental intermediate Jacobian $J_{\mathrm{tr}}^3(X \times C)$.
\begin{cor} \label{ord_comp_prac}
	Let $C$ be a constant cycle curve on a K3 surface $X$, and  $[Z_C] \in \ch^2(X \times C)_{\lhom}$ be a torsion class that lifts $\kappa_C \in \ch^2(X \times k(C))_{0}$. Then we have	 $$\ord(C) \coloneqq \ord(\kappa_C) = \ord(\Phi_{X}^{\mathrm{tr}}([\kappa_C])) = \ord((p_{\mathrm{tr}}  \circ\Phi_{X})[Z_C]),$$ which  by the diagram \eqref{diag_algtrIJ} is simply the minimal positive integer $N$ such that $N[Z_C]$ lies in the image of  $ \ch^1(X) \otimes \ch^1(C)_{\lhom} \ra \ch^2(X \times C)_{\lhom}$.
\end{cor}

\begin{rmk} \label{sum_strategy}
	Note that if a constant cycle curve $C \subset X$ has $\ord(C) \mid N$, then there always exists an $N$-torsion class  $[Z_C] \in \ch^2(X \times C)_{\lhom}$ that lifts $\kappa_C$. Indeed,  one can choose any point $y \in C$  and consider the modified diagonal class
	\begin{equation} \label{compactification}
		[Z_C] \coloneqq [\Delta_C] - [C] \times y -c_X \times [C] \in \ch^2(X \times C)_{\lhom}.
	\end{equation}
By \Cref{def_ord_equiv}, $N[Z_C]$ comes from $\ch^1(X) \otimes \ch^1(C)_{\lhom}$ and is not necessarily zero. However, since  $\ch^1(C)_{\lhom}$ is divisible, one can always modify the class $[Z_C]$ further by a cycle from $\ch^1(X) \otimes \ch^1(C)_{\lhom}$  to obtain an $N$-torsion class, which still lifts $\kappa_{t}^{\prime}$.

To sum up, an effective three-step strategy to compute $\ord(C)$ in $J_{\mathrm{tr}}^3(X \times C)$ is as follows: First, we find a torsion compactification $[Z_C] \in \ch^2(X \times C)_{\lhom}$ of $\kappa_C$, and then compute the image of $[Z_C]$ under the Abel--Jacobi map $\Phi_X$. Finally, we determine the order of the projection of $\Phi_X([Z_C])$ in $J_{\mathrm{tr}}^3(X \times C)$.
\end{rmk}

\begin{rmk} \label{tr_AJ_absurface}
The above construction of the transcendental Abel--Jacobi map can be extended to a surface $S$ with $H^1(S,\mbz) \neq 0$ (e.g.\ an abelian surface) as follows.  For any smooth integral curve $C$, as before, one defines \begin{align} 
	J_{\mathrm{alg}}^3(S \times C) \coloneqq \frac{F^2(\ns(S) \otimes H^1(C))^{\ast}}{\ns(S) \otimes H^1(C,\mbz)} \simeq J(C)^{\oplus \rho(S)} \ \ \textup{and} \ \  
	J_{\mathrm{tr}}^3(S \times C) \coloneqq  \frac{J^3(S \times C)}{J_{\mathrm{alg}}^3(S \times C)},
\end{align}
which fit into the commutative diagram \eqref{diag_algtrIJ}, and we have a well-defined transcendental Abel--Jacobi map
$\Phi_S^{\mathrm{tr}} \colon \ch^2(S \times k(C))_{0} \ra J_{\mathrm{tr}}^3(S \times C)$. 
 The same argument in \Cref{inj_tr_AJ} shows that $\Phi_S^{\mathrm{tr}}$ is  injective on the torsion subgroup. 
Note that in this case, the intermediate Jacobian $J^3(S \times C)$ splits into three direct summands corresponding to the Künneth decomposition of $H_3(S \times C,\mbz)$, and the intermediate Jacobian  \begin{align*}
	J_{\mathrm{tr}}^{2,1}(S \times C) \coloneqq \frac{F^2(T'(S) \otimes H^1(C))^{\ast}}{T'(S) \otimes H_1(C,\mbz)}, \ \ \textup{where} \ \ T'(S) \coloneqq \frac{H^2(S,\mbz)}{\textup{NS}(S)}
\end{align*} associated with the Hodge structure $T'(S) \otimes H^1(C)$ is a proper subtorus of $J_{\mathrm{tr}}^{3}(S \times C)$.
\end{rmk}

\section{Improving the upper bound for $\ord(E_t)$} \label{sec_order_Zt}
 Let  $t \in E_1$ be a torsion point of order $n > 2$, and let $E_t \subset X = \textup{Kum}(E_1 \times E_2)$ be the smooth elliptic constant cycle curve in  \Cref{order_factor}(3) with $\ord(E_t) \mid n$. The goal of this section is to improve this upper bound for even $n$ to $\ord(E_t) \mid \frac{n}{2}$. 
 
The main idea is to use the construction of the Kummer surface $X$  to control $\ord(E_t)$.  By \Cref{sec_review}, the class $\kappa_t \coloneqq (\Delta_{E_t} - c_X \times [E_t]) \mid_{X \times k(E_t)} \in \ch^2(X \times k(E_t))$ is torsion, and by definition  $\ord(E_t) = \ord(\kappa_{t})$. Following the strategy in  \Cref{sum_strategy}, it suffices to find a torsion compactification $[Z_t] \in \ch^2(X \times E_t)$ of $\kappa_t$ and determines its order in $J_{\mathrm{tr}}^3(X \times C)$.

Our candidate for the torsion compactification of $\kappa_t$ is
\begin{equation*} 
	[Z_{t}] \coloneqq [\Delta_{E_t}] -[ E_t \times e] - c_X \times [E_t] \in \ch^2(X \times E_t)_{\lhom}.
\end{equation*} A priori $[Z_t]$ is not torsion, and  computing its image under the Abel--Jacobi map is already a difficult task. However,  the following diagram, induced by the construction of $X$,
\begin{equation} \label{Kummer_diagram}
 	X \times E_t \xleftarrow{\pi} \Bl(E_1 \times E_2) \times E_t \xrightarrow{\phi}
 	E_1 \times E_2 \times E_t,
 \end{equation} enables us to identify $\ord(\kappa_t)$ with the order of a torsion class $\kappa_t^{\prime} \in
\ch^2(E_1 \times E_2\times k(E_t))$, see \Cref{order_kappa_same}. The class $\kappa_t^{\prime}$ has a torsion compactification  $[Z_t^{\prime}] \in \ch^2(E_1 \times E_2 \times E_t)_{\lhom}$, and one can compute $\ord([Z_t^{\prime}])$  explicitly in the intermediate Jacobian  $J^3(E_1 \times E_2 \times E_t)$ (see \Cref{order_CycleInProduct}). Then $\ord(E_t) = \ord(\kappa_t) = \ord(\kappa_t^{\prime}) \mid \ord([Z_t^{\prime}])$ gives the claimed upper bound.

\subsection{} \label{subsec_decomp_cycle}
 The main  input of our computation is the following   decomposition of one-cycles on the product $E_1 \times E_2$ of elliptic curves:
\begin{equation} \label{prod_decomposition}
	\ch^1(E_1 \times E_2) \simeq \ch^1(E_1) \oplus \ch^1(E_2) \oplus \textup{Hom}(J(E_1),J(E_2)),
\end{equation} 
 where the first two direct summands are given by pullbacks of zero cycles on $E_1,E_2$ respectively, and the projection of $\alpha \in \ch^1(E_1 \times E_2)$ to $\textup{Hom}(J(E_1),J(E_2))$ is its induced morphism on Jacobians. 	Note that the cycle classes of the three direct summands of  \eqref{prod_decomposition} are contained in  \begin{center}
	$H^2(E_1) \otimes H^0(E_2), H^0(E_1) \otimes H^2(E_2)$ and $\big(H^1(E_1) \otimes H^1(E_2)\big) \cap H^{1,1}(E_1 \times E_2)$
\end{center} respectively. In particular, under the intersection pairing, the direct summand $\textup{Hom}(J(E_1),J(E_2))$ is orthogonal to the other two summands.

For example, for an isogeny $f \colon E_1 \ra E_2$, its graph $\Gamma_f \subset E_1 \times E_2$ decomposes as  
	\begin{equation} \label{decomp_graph}
		[\Gamma_f]  = [E_1 \times e_2] + \deg(f) \cdot [e_1 \times E_2] + [f] \in \ch^1(E_1 \times E_2),
	\end{equation}
	 where $e_i \in E_i$ is the origin and $[f] \in \textup{Hom}(J(E_1), J(E_2))$ denotes the isogeny induced by $f$, viewed as a one-cycle class via the isomorphism \eqref{prod_decomposition}. In particular, for any elliptic curve $E$,  the graphs of $\textup{id}_E, -\textup{id}_{E}  \in \textup{End}(E)$ are the diagonal $\Delta$ and the antidiagonal $\Delta^{-}$ respectively, and they  decompose as
	\begin{equation} \label{diag_decomp}
	[\Delta] = [e \times  E] + [E \times e]   + [\textup{id}], \ \ 
		 [\Delta^{-}] = [e \times  E] + [E \times e]   + [-\textup{id}] \in \ch^1(E \times E).
	\end{equation}  

\subsection{} 
We denote by $e_i \in E_i$ the origin for $i = 1,2$, and choose $e \in E_t$ to be the point mapped to $e_2$ under the natural isomorphism $E_t \simeq E_2$. Consider the following homologically trivial class
\begin{equation} \label{eq_def_Zt}
	[Z_{t}] \coloneqq [\Delta_{E_t}] -[ E_t \times e] - [x_0 \times E_t] \in \ch^2(X \times E_t)_{\lhom}.
\end{equation}
Here we choose $x_0 \in X$ to be the image of $(t,e_2)$  under the quotient $\Bl(E_1 \times E_2) \ra X$. Since $x_0$ lies on the constant cycle curve $E_t$, we have $[x_0] = c_X$ in $\ch_0(X)$.  

The following result allows us to simplify the application of the strategy in \Cref{sec_strategy} by passing to the abelian threefold $E_1 \times E_2 \times E_2$, as explained at the beginning of \Cref{sec_order_Zt}. 
\begin{lem} \label{order_CycleInProduct}  
	Let $t \in E_1$ be a torsion point of order $n >2$ and $[Z_t] \in \ch^2(X \times E_t)$ be as above. Consider the following pullbacks of Chow groups induced by \eqref{Kummer_diagram}:
	\begin{equation*}
		\ch^2(X \times E_t) \xrightarrow{\pi^{\ast}} \ch^2 (\Bl(E_1 \times E_2)\times E_t)) \xleftarrow{\phi^{\ast}}\ch^2 (E_1 \times E_2 \times E_t).
	\end{equation*} 
	\begin{enumerate}[\normalfont(a)]
		\item The  torsion class  $[Z_{t}^{\prime}] \coloneqq 2 ([t]-[e_1]) \times [\id] \in \ch^2(E_1 \times E_2 \times E_t)_{\lhom}$ satisfies  $\pi^{\ast}[Z_{t}] = \phi^{\ast}[Z_{t}^{\prime}]$.
		\item The order of $[Z_{t}^{\prime}]$ is
		$d(n) \coloneqq
		\begin{cases}
			n / 2,\quad & 2 \mid n \\
			n,\quad & 2 \nmid n.
		\end{cases}
		$
	\end{enumerate} 
\end{lem}

\begin{proof}
	\noindent\textbf{(a)} We identify $E_2 \simeq E_t$. Let  $\Delta \subset E_2 \times E_t$ be the diagonal and $\Delta^{-} \coloneqq \{ (x,-x) \colon x \in E_{2} \}$ be the antidiagonal.  Then it follows from \eqref{diag_decomp} that 
	\begin{align*}
	\ \	\pi^{\ast}[Z_{t}] &= 	\pi^{\ast}([\Delta_{E_t}] - [E_{2} \times e] - [x_0 \times E_t]) \\
		&=\phi^{\ast} \big([t] \times ([\Delta]- [E_{2} \times e] - [e_2 \times  E_t])  +   [-t] \times ([\Delta^{-}]- [E_{2} \times e] - [e_2 \times  E_t])\big) \\
		&\overset{(\ast)}{=} \phi^{\ast} \big([t] \times [\id] \big) -  \phi^{\ast} \big( (2[e_1] - [t]) \times [\id] \big) \\
		&=  \phi^{\ast} [Z_t^{\prime}] \in \ch^2(\Bl(E_1 \times E_2) \times E_t).
	\end{align*} 
	Here, for the equality ($\ast$), we use the relations  $[\id] = - [-\id] \in \Hom(J(E_1),J(E_2))$, and  $[t] + [-t] = 2[e_1] \in \ch^1(E_1)$.

	\noindent\textbf{(b)} 
	By the injectivity of Abel--Jacobi maps on torsion  classes of codimension two (\cite{CTSS83AJcodimtwo}),  it is equivalent to show that the image of $[Z_t^{\prime}]$ under the Abel--Jacobi map \begin{equation*}
		\Phi \colon \ch^2(E_1 \times E_2 \times E_t)_{\lhom} \ra J^3(E_1 \times E_2 \times E_t) = \frac{F^2 H^3(E_1 \times E_2 \times E_t)^{\ast}}{H_3(E_1 \times E_2 \times E_t, \mbz)}
	\end{equation*}
	has order $d(n)$. We   compute the image $\Phi([Z_t^{\prime}]) = \Phi(2([t]-[e]) \times [\id]) $ as follows:
	
	By construction, $t \in E_1$ is a torsion point of order $n$, hence the class $ 2([t]-[e_1]) \in \ch^1(E_1)$  is  torsion of order $d(n)$, where $d(n) = n$ for odd $n$, and $n/2$ for even $n$. Since the Abel--Jacobi map $\Phi_{E_1} \colon  \ch^1(E_1)_{\lhom} \ra J(E_1)  = \frac{H^0(E_1,\omega_{E_{1}})^{\ast}}{H_1(E_1,\mbz)}$ is an isomorphism,  there exists a non-divisible one-cycle $\gamma_t \in H_1(E_1,\mbz)$ such that
 	$\Phi_{E_1}(2([t]-[e_1])) = \frac{1}{d(n)}\int_{\gamma_t} \in J(E_1)$. Then it follows from the definition of $\Phi$ that
 	$\Phi([Z_{t}^{\prime}]) =  \frac{1}{d(n)} \int_{\gamma_t \times [\id]}
 	\in  J^3(E_{1} \times E_{2} \times E_t)$.
 	
 	Now note that the class  $[\id] \in H_2(E_2 \times E_t,\mbz)$ has self intersection $-2$, hence is also not divisible, i.e.\  $\frac{1}{m} [\id] \notin H_2(E_2 \times E_t, \mbz)$ for any integer $m > 1$. As a result, the (topological) $3$-cycle $\gamma_t \times [\id] \in H_3(E_{1} \times E_{2} \times E_t, \mbz)$ is not divisible by any integer $1 < m \mid d(n)$ as well. This implies that $\Phi([Z_{t}^{\prime}])$
	is a torsion class of order $d(n)$ and we conclude that $\ord([Z_t^{\prime}]) = d(n)$.
\end{proof}

\begin{rmk}
	 \Cref{order_CycleInProduct}  also implies that the modified diagonal $[Z_t] \in \ch^2(X \times E_t)$ is   a torsion class with $d(n) \mid \ord([Z_t])$: Note that $\phi^{\ast}$ is injective  (see i.e.\ \cite[Prop.~6.7(e)]{fulton1998intersection}), and $\ker(\pi^{\ast})$ is a two-torsion subgroup. However,  it is not clear whether the equality holds due to the nontriviality of $\ker(\pi^{\ast})$: It contains  non-trivial two-torsion classes of the form $[\bar{E_i}] \otimes E_t[2]$, where $\bar{E_i} \subset X$ is one of the 16 rational curves from the blow-up.  
\end{rmk}
\begin{rmk}
	When $4 \nmid n$, one can  prove that the order of $[Z_{t}^{\prime}] \in \ch^2(E_1 \times E_2 \times E_t)$ is exactly $d(n)$  without using the Abel--Jacobi map and the deep injectivity result in \cite{CTSS83AJcodimtwo} as follows. 
	 For any $1 \leqslant d_1 < d(n)$ with $d_1 \mid d(n)$,  since the order of the torsion class $2d_1([t]-[e_1]) \in \ch^1(E_1)$ is $\frac{d(n)}{d_1} \neq 2$ (note that $2 \nmid d(n)$ when $4 \nmid n$), the class  $d_1[Z_{t}^{\prime}]$ induces a nontrivial morphism
	\begin{equation*}
		d_1[Z_{t}^{\prime}]_{\ast} \colon \ch^1(E_2 \times E_t) \ra \ch^1(E_1), \ \ [\id] \mapsto -4d_1([t]-[e_1]) \neq 0,
	\end{equation*}
	hence $d_1[Z_t^{\prime}] \neq 0$, while by construction $d(n)[Z_{t}^{\prime}] = 0$. This proves $\ord([Z_t^{\prime}]) = d(n)$ for $4 \nmid n$.
	
	For $n = 4k$, the above trick fails, because $[kZ_{t}^{\prime}]_{\ast}\big([\id]\big) = -n([t] - [e_1]) = 0$. 
\end{rmk}

\subsection{Passing to $E_1 \times E_2 \times E_t$} \label{sec_pass_tri_prod}
The aim of this section is to reduce the computation of $\ord(E_t)$ on $X \times E_t$ to the abelian threefold $E_1 \times E_2 \times E_t$. To reduce  notations, we write $A \coloneqq E_1 \times E_2$. 

Let $[Z_t] \in \ch^2(X \times E_t)_{\lhom}$ and $ [Z_t^{\prime}] \in \ch^2(A \times E_t)_{\hom}$
be the torsion classes in  \Cref{order_CycleInProduct} that pull back to the same class in $
\ch^2(\textup{Bl}(A) \times E_t)_{hom}$. Define $\kappa_t, \kappa_t^{\prime}$ as their images under the base change to the function field $k(E_t)$:
\begin{align} \label{def_kappa_prime}
	\ch^2(X \times E_t)_{\lhom} \ra \ch^2(X \times k(E_t))_{0}, \ \ [Z_t] \mapsto \kappa_t \\ 
	\ch^2(A \times E_t)_{\lhom} \ra \ch^2(A \times k(E_t))_{0}, \ \ [Z_t^{\prime}] \mapsto \kappa_t^{\prime}.\nonumber
\end{align} 
Recall from \Cref{sec_review} that $\kappa_t$ is torsion  and $\ord(E_t)  = \ord(\kappa_t) \in \ch^2(X \times k(E_t))$. In addition, since $ \ord([Z_t^{\prime}]) = d(n)$ (see \Cref{order_CycleInProduct}),  we have $d(n)\kappa_t^{\prime} = 0$. The two torsion classes $\kappa_t, \kappa_t^{\prime}$ are linked by the following diagram induced from \eqref{Kummer_diagram} (see \Cref{sec_strategy} \eqref{diag_algtrIJ} and \Cref{tr_AJ_absurface} for the construction of the vertical maps):
\begin{equation}  \label{diagram_AJmap}
	\begin{tikzcd}
		\ch^2(X \times k(E_t))_{0} \arrow[d,"\Phi_X^{\mathrm{tr}}"] \arrow[r, "\pi^{\ast}_{\eta}"] 
		& \ch^2 (\Bl(A)\times k(E_t))_{0} \arrow[d,"\Phi_{\Bl}^{\mathrm{tr}}"] 
		& \ch^2 (A \times k(E_t))_{0} \arrow[l,"\phi^{\ast}_{\eta}",swap] \arrow[d,"\Phi_{A}^{\mathrm{tr}}"] \\
		J_{\mathrm{tr}}^3(X \times E_t) \arrow[r,"\pi^{\ast}_{\mathrm{tr}}"]                     & J_{\mathrm{tr}}^3 (\Bl(A)\times E_t)                   
		& J_{\mathrm{tr}}^3 (A \times E_t) \arrow[l,"\phi^{\ast}_{\mathrm{tr}}",swap]    	\end{tikzcd}.
\end{equation}
By construction,  $\pi_{\eta}^{\ast}(\kappa_t) = \phi_{\eta}^{\ast}(\kappa_t^{\prime})$.  We have the following useful observation:
\begin{lem} \label{trAJ_pullback}
	The map $\pi_{\mathrm{tr}}^{\ast}$ is injective and $\phi_{\mathrm{tr}}^{\ast}$ is an isomorphism.
\end{lem}
\begin{proof}
	Since the blow-up  $\Bl(A) \ra A$ preserves the transcendental lattice, the pullback $\phi_{\mathrm{tr}}^{\ast}$ is an isomorphism. Now we show that $\pi_{\mathrm{tr}}^{\ast}$ is injective.
	
	First,  observe that the two pullback morphisms  $\pi^{\ast} \colon J^3(X \times E_t) \ra J^3(\Bl(A) \times E_t)$ and  $\pi_{\mathrm{tr}}^{\ast} \colon J_{\mathrm{tr}}^3(X \times E_t) \ra J_{\mathrm{tr}}^3(\Bl(A) \times E_t)$ induced by the quotient $\pi \colon \Bl(A) \ra X$ have images
	\begin{align*}
	J^{2,1}(\Bl(A) \times E_t)  &\coloneqq  \frac{F^2(H^2(\Bl(A)) \otimes H^1(E_t))^{\ast}}{H^2(\Bl(A),\mbz) \otimes H^1(E_t,\mbz)} \subset J^3(\Bl(A) \times E_t),\\
	J_{\mathrm{tr}}^{2,1}(\Bl(A) \times E_t) &\coloneqq \frac{F^2(T'(A) \otimes H^1(E_t))^{\ast}}{T'(A) \otimes H^1(E_t,\mbz)} \subset J_{\mathrm{tr}}^3(\Bl(A) \times E_t) \ \ \textup{with}  \ \ T'(A) \coloneqq \frac{H^2(S,\mbz)}{\ns(A)},
	\end{align*}
	respectively. Then from \Cref{sec_strategy}\eqref{ext_IJ} we obtain the following pullback diagram of extensions of intermediate Jacobians:
	\begin{equation}  
		\begin{tikzcd}[sep=1.8em, font=\small]
			0 \arrow[r] &J_{\mathrm{alg}}^3(X \times E_t) \arrow[d,"\pi_{\mathrm{alg}}^{\ast}"] \arrow[r] 
			& J^3(X \times E_t) \arrow[r] \arrow[d,"\pi^{\ast}"] 
			& J_{\mathrm{tr}}^3(X \times E_t)  \arrow[d,"\pi_{\mathrm{tr}}^{\ast}"] \arrow[r] &0 \\
			0 \arrow[r] &J_{\mathrm{alg}}^3(\Bl(A) \times E_t)  \arrow[r]     & J^{2,1}(\Bl(A) \times E_t) \arrow[r]                
			& J_{\mathrm{tr}}^{2,1}(\Bl(A) \times E_t)   \arrow[r] &0  	\end{tikzcd}.
	\end{equation} 
	By the snake lemma, to see that $\pi_{\mathrm{tr}}^{\ast}$ is injective, it is equivalent to show  that  $\ker(\pi_{\mathrm{alg}}^{\ast}) \ra \ker(\pi^{\ast})$  and $\pi_{\mathrm{alg}}^{\ast}$ are both surjective. 
	
	Recall from \Cref{sec_strategy}\eqref{diag_algtrIJ} that for the two surfaces $X$ and $\Bl(A)$, we have isomorphisms  $ J_{\mathrm{alg}}^3(\cdot \times E_t) \simeq \ns(\cdot) \otimes   \ch^1(E_t)_{\lhom}$, and  under this identification $\pi_{\mathrm{alg}}^{\ast} = \pi_{NS}^{\ast} \otimes \id$, where $\pi_{NS}^{\ast} \colon \ns(X) \ra \ns(A)$ is the pullback.   Since $\ch^1(E_t)_{\lhom}$ is divisible,  for any $\theta \in \ch^1(E_t)_{\lhom}$, there exists $\alpha \in \ch^1(E_t)_{\lhom}$ such that $\theta = 2 \alpha$. Therefore, for any  $D \in \ns(\Bl(A))$, we have
	$D \otimes \theta = 	 2D \otimes \alpha =  \pi^{\ast}(\pi_{\ast}(D)) \otimes \alpha = \pi_{\mathrm{alg}}^{\ast}(D \otimes \alpha) \in \textup{im}(\pi_{\mathrm{alg}}^{\ast}) ,
	$
	which implies that $\pi_{\mathrm{alg}}^{\ast}$ is surjective.
	
	 Now it remains to prove the surjectivity of $\ker(\pi_{\mathrm{alg}}^{\ast}) \ra \ker(\pi^{\ast})$. By construction, we have  $\ker(\pi^{\ast}) \simeq \frac{H^2(\Bl(A),\mbz)}{\pi^{\ast}H^2(X,\mbz)} \otimes H^1(E_t,\mbz)$. Denote by $E_i \subset \Bl(A)$  the 16 exceptional curves and write $\bar{E_i} \subset X$ for the image of $E_i$. Let $K \subset H^2(X,\mbz)$ be the Kummer lattice, i.e.\ the minimal primitive sublattice that contains  $\oplus_i \mbz [\bar{E_i}]$. Recall that $K$ is of rank $16$ with  discriminant $2^6$ (see e.g. \cite{BHPV2003-cptsurface}). Since the pullback $\pi^{\ast} \colon H^2(X,\mbz) \ra H^2(\Bl(A),\mbz)$ is an injective map of free $\mbz$-modules of the same rank,  one can view $\pi^{\ast}H^2(X,\mbz) \subset H^2(\Bl(A),\mbz)$ and $\pi^{\ast}(K) \subset \oplus \mbz [E_i]$ as sublattices and check that both embeddings are of index $2^{11}$. Also note that  by definition $K = (\pi^{\ast})^{-1}(\oplus \mbz [E_i])$. Therefore, the two embeddings of lattices induce a natural isomorphism
	$\frac{\oplus \mbz [E_i]}{\pi^{\ast}(K)} \simeq \frac{H^2(\Bl(A),\mbz)}{\pi^{\ast}H^2(X,\mbz)},
$
	and hence $$\ker(\pi^{\ast}) \simeq 	\frac{\oplus \mbz [E_i]}{\pi^{\ast}(K)} \otimes H^1(E_t,\mbz) \simeq \frac{\oplus \frac{1}{2} \mbz [\bar{E_i}]}{K} \otimes H^1(E_t,\mbz) \subset J^3(X \times E_t)[2].
	$$
	Let $\Phi_X$ be the Abel--Jacobi map on $\ch^2(X \times C)_{\lhom}$. Then $\ker(\pi^{\ast}) \subset \Phi_X(\oplus_i \mbz [\bar{E_i}] \otimes J(E_t)[2])$ is contained in $J_{\mathrm{alg}}^3(X \times E_t)$ and   $\pi_{\mathrm{alg}}^{\ast}(\ker(\pi^{\ast})) = 0$. 
	This concludes the proof.
\end{proof}
 
  The following result says that the order of the constant cycle curve $E_t$ can be computed as the order of the torsion class $\kappa_t^{\prime} \in \ch^2(A \times k(E_t))$.
\begin{prop} \label{order_kappa_same} 
 Let $t \in E_1$ be a torsion point of order $n>2$. Then for the two torsion classes $\kappa_t \in \ch^2(X \times k(E_t))$ and $\kappa_t^{\prime} \in \ch^2(A \times k(E_t))$ in  \eqref{def_kappa_prime}, we have  $$\ord(E_t) = \ord(\kappa_t) =  \ord(\kappa_t^{\prime}) \mid d(n).$$
\end{prop}

\begin{proof}
	By construction (cf.\ \eqref{def_kappa_prime}), we have $d(n)\kappa_t^{\prime} = 0$ (see \Cref{order_CycleInProduct}). We use the diagram \eqref{diagram_AJmap} to compare the order of $\kappa_t$ and $\kappa_t^{\prime}$. Recall from   \Cref{inj_tr_AJ} and \Cref{tr_AJ_absurface} that $\Phi_X^{\mathrm{tr}}, \Phi_{Bl}^{\mathrm{tr}}$ and $\Phi_{A}^{\mathrm{tr}}$ are all injective on the torsion subgroups. In addition, by \Cref{trAJ_pullback}, $\pi_{\mathrm{tr}}^{\ast}$ is injective and $\phi_{\mathrm{tr}}^{\ast}$ is an isomorphism. Therefore, we conclude that $\ord(\kappa_t) = \ord(\kappa_t^{\prime}) \mid d(n)$. 
\end{proof}

\section{The order of $E_{t}$} \label{sec_ord_kprime}
Our final goal is to determine the order of the elliptic constant cycle curve $E_t$ on the Kummer surface $X = \textup{Kum}(E_1 \times E_2)$ for any torsion point $t \in E_1$ of order $n > 2$ (see \Cref{order_factor}(3)). 

In \Cref{sec_order_Zt}, we constructed a $d(n)$-torsion class $\kappa_t^{\prime} \in \ch^2(E_1 \times E_2 \times k(E_t))$ (see \eqref{def_kappa_prime}) and proved that $\ord(E_t) = \ord(\kappa_t^{\prime})$ (see \Cref{order_kappa_same}). The aim of this section is to show that  $\ord(\kappa_t^{\prime}) = d(n)$ under mild assumptions on the triple $(E_1,E_2,n)$, which completes the proof of our main result \Cref{ccc_Kummer_ord}, namely  $\ord(E_t) = d(n)$ holds in the generic case. 

The idea to prove $\ord(\kappa_t^{\prime}) = d(n)$ is as follows. Since   $\kappa_t^{\prime}$ is the restriction of the torsion class $[Z_t^{\prime}]  = 2([t] - [e_1]) \times [\id] \in \ch^2(E_1 \times E_2 \times E_t)$ (cf.\ \eqref{def_kappa_prime}), by \Cref{lem_bloch},  it suffices to show that the minimal positive integer $N$ such that  $N[Z_t^{\prime}]$ comes from  $\ch^1(E_1 \times E_2) \otimes \ch^1(E_t)$ is $d(n)$.  
When $E_1 \nsim E_2$,  the minimality of $d(n)$ follows immediately from a detailed study of product cycles in $\ch^2(E_1 \times E_2 \times E_t)$, see \Cref{sec_decomp_triple_elliptic} and \Cref{sec_case1}.  
The case $E_1 \sim E_2$ is more complicated. Nonetheless, in the general situation, i.e.\ $E_1 \sim E_2$ with no CM, one can  still verify that $\ord(\kappa_t^{\prime})$ does not degenerate by studying the image of $[Z_t^{\prime}]$ in the intermediate Jacobian $J^3(E_1 \times E_2 \times E_t)$, see \Cref{sec_case2}.

\subsection{One-cycles on triple products of elliptic curves} \label{sec_decomp_triple_elliptic}
In this subsection, we work in a slightly more general setting.
 Let $E_1, E_2,E_3$ be three elliptic curves. Our goal is to understand the structure of the subgroup of  $\ch^2(E_1 \times E_2 \times E_3)$ generated by product cycle classes, namely classes in 
$G_{ij} \coloneqq \textup{im}(\ch^1(E_i \times E_j) \otimes \ch^1(E_k) \ra \ch^2(E_1 \times E_2 \times E_3))$ for  $\{i,j,k\} = \{1,2,3\}$. Due to  symmetry, it is enough to study the two groups $G_{12}$, $G_{23}$ and their intersection. 

By the decomposition of $\ch^1(E_i \times E_j)$ (see \Cref{subsec_decomp_cycle}),
one can decompose $G_{12}$ and $G_{23}$ further as
\begin{equation} \label{eq_Gij_decomp}
	G_{12} = G^{0,2,2} + G^{1,1,2} + G^{2,0,2} , \ \  
	G_{23} = G^{2,0,2}+G^{2,1,1}+G^{2,2,0},
\end{equation}
where 
\begin{align} \label{eq_Gpqr}
	G^{0,2,2} &\coloneqq \textup{im}\big( [E_1] \otimes \ch^1(E_2) \otimes \ch^1(E_3)  \ra \ch^2(E_1 \times E_2 \times E_3)\big) \nonumber \\
	G^{1,1,2} &\coloneqq \textup{im}\big( \textup{Hom}(J(E_1),J(E_2)) \otimes \ch^1(E_3)  \ra \ch^2(E_1 \times E_2 \times E_3)\big) \nonumber \\
	G^{2,0,2} &\coloneqq \textup{im}\big(\ch^1(E_1) \otimes [E_2] \otimes \ch^1(E_3)  \ra \ch^2(E_1 \times E_2 \times E_3)\big) \\
	G^{2,1,1} &\coloneqq \textup{im}\big( \ch^1(E_1) \otimes \textup{Hom}(J(E_2),J(E_3))    \ra \ch^2(E_1 \times E_2 \times E_3)\big) \nonumber \\
	G^{2,2,0} &\coloneqq \textup{im}\big(\ch^1(E_1) \otimes  \ch^1(E_2) \otimes [E_3]  \ra \ch^2(E_1 \times E_2 \times E_3)\big), \nonumber 
\end{align}
and the superscript $(p,q,r)$ indicates that the corresponding cohomology classes are of the type $H^p(E_1) \otimes H^q(E_2) \otimes H^q(E_3)$. 

The following lemma describes the  structure of the subgroup $G_{12}+G_{23} \subset \ch^1(E_1 \times E_2 \times E_3)$. Note  that $G^{2,2,0} \cap G_{12} = 0$ and $G^{2,0,2} \subset G_{12} \cap G_{23}$, hence the only interesting part of the intersection $G_{12} \cap G_{23}$ is contained in $G^{2,1,1}$.

\begin{lem} \label{chow_tripleprod}
	Let $E_1, E_2, E_3$ be three elliptic curves. The  subgroups  of $\ch^2(E_1 \times E_2 \times E_3)$ in \eqref{eq_Gij_decomp} and  \eqref{eq_Gpqr} satisfy:
	\begin{enumerate}[\normalfont(a)]
		\item $G_{12} =  G^{0,2,2} \oplus G^{1,1,2}\oplus G^{2,0,2} $.
		\item $G_{12} \cap G^{2,1,1}  \subset G^{1,1,2}$. 
		In particular, if  $E_1 \nsim E_2$, then $G_{12} \cap G^{2,1,1} = 0$.
	\end{enumerate}
\end{lem}

\begin{proof}
	We first prove (b). Any  $\alpha \in G_{12} =   G^{0,2,2} + G^{1,1,2}+G^{2,0,2}$ can be represented as 
	\begin{align} \label{intersect_directsum}
		\alpha =	\sum_{j} [E_1] \times \beta_{j}  \times \theta_{j}+\sum_{i} \alpha_{i} \times [E_2] \times \gamma_{i}   + \sum_{k} D_k  \times \eta_k  \in \ch^2(E_1 \times E_2 \times E_3)
	\end{align}
	for some  $\alpha_i \in \ch^1(E_1)$, $\beta_j \in \ch^1(E_2)$, $\gamma_i, \theta_j, \eta_k  \in \ch^1(E_3)$  and $D_{k} \in \textup{Hom}(J(E_1),J(E_2))$.  Denote by $e_i \in E_i$ the origin. If we assume further $\alpha \in G^{2,1,1}$, then  
	$$\big(\alpha.[E_1] \times [e_2] \times [E_3]\big) = 0 = \big(\alpha.\ [e_1] \times [E_{2}] \times [E_3]\big)$$
	and this implies
	$
	\sum_{i} \alpha_{i} \times [e_2] \times \gamma_{i} = 0 = \sum_{j} [e_{1}] \times \beta_{j} \times \theta_{j}.
	$
	By pushing  forward the two equations to $E_1 \times E_3$ and $E_2 \times E_3$ respectively, we get
	\begin{equation*}
		0 = \sum_{i} \alpha_{i} \times\gamma_{i} \in \ch^2(E_1 \times E_3), \ \ \  0 = 
		\sum_{j} \beta_{j} \times \theta_{j} \in \ch^2(E_2 \times E_3).
	\end{equation*}
	Therefore, the first two summands in \eqref{intersect_directsum} are trivial and $\alpha = \sum_{k} D_k  \times \eta_k  \in G^{1,1,2}$. 
	
	Since $\alpha \in G_{12} \cap G^{2,1,1}$ is arbitrary in the above argument, we conclude  that $G_{12} \cap G^{2,1,1}$ is contained in $G^{1,1,2}$. In particular, if $E_1 \nsim E_2$, then  $G^{1,1,2} = 0$, which implies $G_{12} \cap G^{2,1,1} = 0$.
	
	For (a), it remains to show that $G_{12} =  G^{0,2,2} + G^{1,1,2} +G^{2,0,2}$ is a direct sum. By setting $\alpha = 0$ in \eqref{intersect_directsum}, the above argument still works and one concludes that the three summands in \eqref{intersect_directsum} have to be trivial, hence $G_{12} =   G^{0,2,2} \oplus G^{1,1,2}\oplus G^{2,0,2}$.
\end{proof}

	When $E_1 \sim E_2$, the group $G^{1,1,2}$ is nontrivial and  $G_{12} \cap G^{2,1,1}$ is not necessarily trivial. We refer to \Cref{counterexample} for a counterexample. However, for our purpose, it would be enough to know whether a certain torsion class $Z$ in $G^{2,1,1}$ belongs to $G_{12}$. If we know in addition that $Z \notin G^{1,1,2}$, then the above lemma implies immediately that $Z \notin G_{12}$.
	
	The following result is useful to test whether a homologically trivial class $Z$ lies in $G^{1,1,2}$. 
\begin{lem} \label{ker_sum}
	Let $E_1,E_2$ be two elliptic curves and $\Sigma \colon E_2 \times E_2 \ra E_2$  be the addition map.  Consider the pushforward 
$(\textup{id}_{E_1} \times \Sigma)_{\ast} \colon \ch^2(E_1 \times E_2 \times E_2)  \ra \ch^1(E_1 \times E_2)
$. Then    $$(\textup{id}_{E_1} \times \Sigma)_{\ast} ([g] \otimes \beta) = \deg(\beta) \cdot [g] $$
for any $g \in \textup{Hom}(J(E_1),J(E_2))$ and $\beta \in \ch^1(E_2)$.
	In particular,  the homologically trivial subgroup \begin{equation}
		G_{\lhom}^{1,1,2} \coloneqq \textup{im}\big( \textup{Hom}(J(E_1),J(E_2)) \otimes \ch^1(E_3)_{\hom}  \ra \ch^2(E_1 \times E_2 \times E_3)\big) \subset G^{1,1,2}
	\end{equation} is contained in $\ker(\textup{id}_{E_1} \times \Sigma)_{\ast}$. Note that the conclusion is automatically true when $E_1 \nsim E_2$.
\end{lem}
\begin{proof}
	We abuse the notation and  still denote by $g \colon E_1 \ra E_2$ the isogeny that induces the map $g \in \textup{Hom}(J(E_1),J(E_2))$.  For  any point $y \in E_2$, denote by $g \oplus y \colon E_1 \ra E_2 \colon x \mapsto g(x) \oplus y$ the translation of $g$ by the point $y$ on the image. Then the two graphs $\Gamma_g, \Gamma_{g \oplus y}$ decompose as	\begin{align} 
		[\Gamma_{g}] &=  [E_1 \times e_2]+ \deg(g) \cdot [e_1 \times E_2]    + [g],\\
		[\Gamma_{g \oplus y}] &=  [E_1 \times y] + \deg(g) \cdot [e_1 \times E_2]    + [g]  \ \ \ \ \textup{in} \   \ch^1(E_1 \times E_2), \nonumber
	\end{align} 
	and for any $\beta = \sum_{i} a_i[y_i] \in \ch^1(E_2)$, we have
	 \begin{align*}
			(\textup{id}_{E_1} \times \Sigma)_{\ast}([g] \otimes \beta) &= (\textup{id}_{E_1} \times \Sigma)_{\ast} \big[[\Gamma_g] \otimes \beta - [E_1 \times e_2] \otimes \beta - \deg(g) \cdot [e_1 \times E_2] \otimes \beta)\big] \\
			&= \sum_i a_i[\Gamma_{g \oplus y_i}] -[E_1] \otimes \beta - \deg(\beta) \cdot \deg(g) \cdot [e_1 \times E_2] \\
			&= \sum_{i} a_i \big([E_1] \times [y_i] + \deg(g) \cdot [e_1 \times E_2]    + [g] \big) 
			-[E_1] \otimes \beta  \\
			& \qquad \qquad \qquad \qquad \ \ - \deg(\beta) \cdot \deg(g) \cdot [e_1 \times E_2]\\
			&= \sum_i a_i \cdot [g] = \deg(\beta) [g].
		\end{align*}	
\end{proof}
 
\subsection{Case 1: $E_1 \nsim E_2$ or $4 \nmid n$} \label{sec_case1}
We apply the observations in \Cref{sec_decomp_triple_elliptic} to show that $\ord(\kappa_t^{\prime}) = d(n)$. The following result covers the generic cases of our main result  \Cref{ccc_Kummer_ord}.
\begin{prop}\label{order_kappa'_case1}
	Let $n>2$ be an integer and $t \in E_1$ be a torsion point of order $n$. If either $E_1 \nsim E_2$ or $4 \nmid n$, then the  constant cycle curve $E_t \subset X = \textup{Kum}(E_1 \times E_2)$  has 
	$\ord(E_t) = d(n)$.
\end{prop}

\begin{proof}
Let  $[Z_t^{\prime}] = 2 ([t]-[e_1]) \times [\id]  \in \ch^2(E_1 \times E_2 \times E_t)$ be the torsion class  of order $d(n)$  in \Cref{order_CycleInProduct}. By the discussion at the beginning of \Cref{sec_ord_kprime}, it suffices to prove that the minimal positive integer $N$ such that $$N[Z_t^{\prime}] \in G_{12} \coloneqq \textup{im}(\ch^1(E_1 \times E_2) \otimes \ch^1(E_t) \ra \ch^2(E_1 \times E_2 \times E_t))$$
is $d(n)$. Since $d(n)[Z_t^{\prime}] = 0 \in G_{12}$ by construction, it remains to show that $d(n)$ is minimal.

	We argue by contradiction. Assume that  there exists a factor $d_1 \mid d(n)$ with $1 \leqslant d_1 < d(n)$  such that $d_1[Z_t^{\prime}] \in G_{12}$. By definition, $d_1[Z_t^{\prime}]$ also lies in the subgroup
		 \begin{equation*} 	 G_{\lhom}^{2,1,1} \coloneqq \textup{im}\big(\ch^1(E_1)_{\lhom} \otimes \textup{Hom}(J(E_2),J(E_t)) \ra \ch^2(E_1 \times E_2 \times E_t)\big).
		 \end{equation*}
 Therefore, restricting
 \Cref{chow_tripleprod} to homologically trivial classes yields    \begin{equation*}
	d_1[Z_t^{\prime}] \in G_{\lhom}^{2,1,1} \cap G_{12} \subset G_{\lhom}^{1,1,2} = \textup{im}\big( \textup{Hom}(J(E_1),J(E_2)) \otimes \ch^1(E_t)_{\lhom}  \ra \ch^2(E_1 \times E_2 \times E_t)\big). 
 \end{equation*}
We now show that this leads to a contradiction under either of the two conditions stated, hence the  assumption fails. \\
(a)\ If $E_1 \nsim E_2$, then $\textup{Hom}(J(E_1),J(E_2)) = 0$ and $ G_{\lhom}^{1,1,2}= 0$.  However, since $\ord([Z_t^{\prime}]) = d(n)$, the class  $d_1[Z_t^{\prime}] \neq 0$, thus cannot be contained in $G_{\lhom}^{1,1,2}$, a contradiction.\\
(b)\ Now we only require $4 \nmid n$. Let $\Sigma \colon E_2 \times E_t \ra E_2$ be the addition map (we identify $E_2 \simeq E_t$) and $\textup{id}_{E_1} \times \Sigma \colon E_1 \times E_2 \times E_t \ra E_1 \times E_2$ be the product map. Recall from \Cref{ker_sum} that $G_{\lhom}^{1,1,2}$ is contained in the kernel of the induced map $(\id_{E_1} \times \Sigma)_{\ast} \colon \ch^2(E_1 \times E_2 \times E_t) \ra \ch^1(E_1 \times E_2)$. We show that  $d_1[Z_t^{\prime}] \notin \ker(\id_{E_1} \times \Sigma)_{\ast}$, which again leads to a contradiction to  $d_1[Z_t^{\prime}] \in G_{\lhom}^{1,1,2}$.

	To verify this, one  computes that	$(\textup{id}_{E_1} \times \Sigma)_{\ast}\big(d_1[Z_{t}^{\prime}]\big) =  4d_1([t]-[e_1]) \times [E_2] \in  \ch^1(E_1 \times E_2)$.
	Consider the morphism induced by $4d_1([t]-[e_1]) \times [E_2]$ as a correspondence: $$\ch^2(E_2) \ra \ch^1(E_1), \ \ \  [e_2] \mapsto 4d_1([t]-[e_1]).$$
		Since $t \in E_1$ is a torsion point of order $n$ and $4 \nmid n$, the torsion class $2d_1([t]-[e_1]) \in \ch^1(E_1)$ is of odd order $\frac{d(n)}{d_1}$, hence $4d_1([t]-[e_1]) \neq 0$ and the class $4d_1([t] - [e_1]) \times [E_1]$ is nontrivial itself. Therefore, $d_1[Z_t^{\prime}] \notin \ker(\id_{E_1} \times \Sigma)_{\ast}$ and it cannot be contained in $G_{\lhom}^{1,1,2}$.
 
	To sum up, when $E_1 \nsim E_2$ or $4 \nmid n$, the positive integer $d(n)$ is minimal with the property that $d(n)[Z_t^{\prime}] \in G_{12}$, hence we conclude  that $\ord(E_t) = d(n)$.
\end{proof}

\subsection{Case 2: $E_1 \sim E_2$ and $4 \mid n$} \label{sec_case2}
The Chow group argument from  \Cref{sec_case1} no longer applies to this case. However, one can still compute $ \ord(\kappa_t^{\prime})$ in  $J^3(E_1 \times E_2 \times E_t)$. To be explicit, by \Cref{inj_tr_AJ} and \Cref{tr_AJ_absurface}, we have $\ord(\kappa_t^{\prime}) = \ord(\Phi_{\mathrm{tr}}([Z_t^{\prime}]))$, and this can be determined by  comparing the image of   $[Z_t^{\prime}]$ with classes from $\ch^1(E_1 \times E_2) \otimes \ch^1(E_t)_{\lhom}$ under the Abel--Jacobi map $\Phi \colon \ch^2(E_1 \times E_2 \times E_t)_{\lhom} \ra J^3(E_1 \times E_2 \times E_t)$.  The  key ingredient  is again the injectivity  of Abel--Jacobi maps on codimension-two torsion classes (\cite{CTSS83AJcodimtwo}).
\begin{prop} \label{ord_4torsion}
	Let $E_1,E_2$ be two isogeneous elliptic curves, $t \in E_1$  be a torsion point of order $n > 2$ and $4 \mid n$. Then the constant cycle curve $E_t \subset X = \textup{Kum}(E_1 \times E_2)$ has order $\frac{n}{2}$, if
	\begin{enumerate}[\normalfont(a)]
		\item either $E_1$ and $E_2$ are isomorphic with CM,
		\item or $E_1$ and $E_2$ have no CM.
	\end{enumerate}  
\end{prop}
When $E_1 \sim E_2$ with CM,  the additional assumption $E_1 \simeq E_2$ cannot be easily removed. See \Cref{counterexample} for an example  where $\ord(E_t)$ can be either $n/2$ or strictly smaller, depending on certain numerical conditions on $E_1$ and $E_2$.

The proof of \Cref{ord_4torsion} reduces to the following lemma.
\begin{lem} \label{case2_lem}
	Let $E_1,E_2$ be two isogeneous elliptic curves. If $E_1$ and $E_2$ have CM, we assume in addition that $E_1 \simeq E_2$. Then for  any nonzero  two-torsion $\alpha \in \ch^1(E_1)$, the torsion class $Z_{\alpha} \coloneqq \alpha \times [\id] \in \ch^2(E_1 \times E_2 \times E_2)$ does not lie in the subgroup
	$$G_{12} \coloneqq \textup{im}\big( \ch^1(E_1 \times E_2) \otimes \ch^1(E_2)  \ra \ch^2(E_1 \times E_2 \times E_2)\big).$$
\end{lem}
\begin{proof}
(a) Consider first the case when $E_1 \simeq E_2$ with CM. Without loss of generality, we identify $E_1$ and $E_2$.   Then $\Hom(J(E_1),J(E_2)) \simeq \mbz \textup{id} \oplus \mbz \phi$, where $\phi$ is not a multiplication by an integer. 
	
		 We prove as in \Cref{order_kappa'_case1} by contradiction. Assume that $Z_{\alpha} \in G_{12}$. Since by definition  \begin{equation*} Z_{\alpha} \in G_{\lhom}^{2,1,1} \coloneqq \textup{im}\big(\ch^1(E_1)_{\lhom} \otimes \textup{Hom}(J(E_2),J(E_2)) \ra \ch^2(E_1 \times E_2 \times E_2)\big),
		\end{equation*} again  \Cref{chow_tripleprod} implies that 
		$$
		Z_{\alpha} \in G_{\lhom}^{1,1,2} =  \textup{im}\big( \textup{Hom}(J(E_1),J(E_2)) \otimes \ch^1(E_1)_{\lhom}  \ra \ch^2(E_1 \times E_2 \times E_2)\big)
	.$$ However, the argument in \Cref{order_kappa'_case1} does not lead to a contradiction to $Z_{\alpha} \in G_{\lhom}^{1,1,2}$ now: When $E_1 \sim E_2$ and $4 \mid n$, the group $G_{\lhom}^{1,1,2} \neq 0$ and $(\id_{E_1} \times \Sigma)_{\ast}(Z_{\alpha}) = 0$.  Our solution  is to compare the image of $Z_\alpha$ and the subgroup $G^{1,1,2}_{\lhom}$ under the Abel--Jacobi map $$\Phi \colon \ch^2(E_1 \times E_2 \times E_2)_{\lhom} \ra J^3(E_1 \times E_2 \times E_2).$$ 
	
	Since $Z_{\alpha} \in G_{\lhom}^{1,1,2}$,  there exists $\beta_1, \beta_2 \in \ch^1(E_2)_{\lhom}$  such that
	\begin{equation} \label{eq_contradiction}
	Z_{\alpha} = \alpha \times [\id] = [\id] \times \beta_1+ [\phi] \times \beta_2 \in \ch^2(E_1 \times E_2 \times E_2).
	\end{equation}
	Here $[\id],[\phi] \in \ch^1(E_1 \times E_2)$ denote the induced classes (see \Cref{subsec_decomp_cycle}). Note that the same argument in  \Cref{order_CycleInProduct}(b) implies that $Z_{\alpha}$ is a nonzero two-torsion class. Since the Abel--Jacobi map $\Phi$ preserves the order of any  codimension-two torsion class (\cite{CTSS83AJcodimtwo}), and $[\id],[\phi]$ are not divisible as generators of $\Hom(J(E_1),J(E_2))$, we see that 
	 $\beta_1,\beta_2 \in \ch^1(E_2)[2]$ and cannot be both trivial. 
	
	Now we compare the images of the two representations of  $Z_{\alpha}$ in \eqref{eq_contradiction} under the  Abel--Jacobi map $\Phi$.
	Since the Abel--Jacobi maps $\Phi_{E_1}, \Phi_{E_2}$ for $E_1,E_2$ are isomorphisms,  for  $\alpha \in \ch^1(E_1)[2]$ and $\beta_1, \beta_2 \in \ch^1(E_2)[2]$,  there exist $\gamma \in H_1(E_1,\mbz)$ and $\gamma_1, \gamma_2 \in H_1(E_2,\mbz)$
such that
\begin{equation*}
	\Phi_{E_1}(\alpha) = \frac{1}{2}\int_{\gamma} \in J(E_1), \ \ \Phi_{E_2}(\beta_1) = \frac{1}{2}\int_{\gamma_1}, \ \ \Phi_{E_2}(\beta_2) = \frac{1}{2}\int_{\gamma_2} \in J(E_2).
\end{equation*}
Then the image of \eqref{eq_contradiction} under $\Phi$ is by definition	\begin{align} \label{eq_mod_lattice}
		\Phi(Z_{\alpha}) = \frac{1}{2} \int_{\gamma \times [\id]} =  \frac{1}{2} \int_{[\id] \times \gamma_1} + \frac{1}{2} \int_{[\phi] \times \gamma_2}\in J^3(E_1 \times E_2 \times E_2).
	\end{align} 
	Let $v_0,v_1 \in H_1(E_1,\mbz)$ and $w_0,w_1 \in H_1(E_2,\mbz)$ be the canonical generators,  with topological intersection numbers $(v_0.v_0) = 0$, $(v_1.v_1)  = 0$ and $(v_0.v_1)= - (v_0.v_1) = 1$; the same relations hold for $w_0,w_1$. By writing
	\begin{equation*} 
		\gamma = b_0v_0 + b_1v_1 \in H_1(E_1,\mbz), \ \ \gamma_1 = c_0w_0 + c_1w_1, \ \ \  \gamma_2 = d_0w_0 + d_1w_1 \in H_1(E_2,\mbz),
	\end{equation*} 
 one reformulates the equation \eqref{eq_mod_lattice} as
\begin{align*}
	(b_0v_0 + b_1v_1) \otimes (-w_0 \otimes w_1 + w_1 \otimes w_0) \equiv &(-v_0 \otimes w_1 + v_1 \otimes w_0) \otimes (c_0w_0+c_1w_1) \\
   & + \big(\sum_{i,j \in \{0,1\}} a_{ij} \cdot v_i \otimes w_j\big) \otimes (d_0w_0 + d_1w_1) \\
   &\qquad \qquad \qquad \quad \textup{mod} \ \ 2H_3(E_1 \times E_2 \times E_2,\mbz),
\end{align*} which reduces to  the following congruence relations among  the integer coefficients:
\begin{align*}
	b_0 &\equiv a_{00}d_1 \equiv a_{01}d_0 + c_0, \ \  a_{00}d_0 \equiv a_{01}d_{1} +c_1 \equiv 0 \pmod 2\\
	b_1 &\equiv a_{11}d_0 \equiv a_{10}d_1 +c_1    , \ \ a_{11}d_{1} \equiv a_{10}d_{0} +c_0 \equiv 0 \pmod 2.
\end{align*} 
One deduces that
\begin{align*}
	&b_0^2 \equiv a_{00}d_1 (a_{01}d_0+c_0) \equiv a_{00}d_1 (a_{01}d_0+a_{10}d_0) \equiv a_{00}d_{0}d_{1}(a_{01}+a_{10})\equiv 0 \pmod 2\\ &b_1^2 \equiv a_{11}d_0 (a_{10}d_1+c_1) \equiv a_{11}d_0 (a_{10}d_1+a_{01}d_1) \equiv a_{11}d_{1}d_{0}(a_{10}+a_{01})\equiv 0 \pmod 2,
\end{align*}
 i.e.\ $b_0$ and $b_1$ are both even. However, since $\alpha \neq 0$ and $\Phi_{E_1}$ is an isomorphism, the image $\Phi_{E_1}(\alpha) = \frac{1}{2}\int_{\gamma} \in J(E_1)$  is also nontrivial. This implies that $2 \nmid \gamma = b_0v_0 + b_1v_1 \in H_1(E_{1}, \mbz)$, namely $2 \nmid gcd(b_0,b_1)$, a contradiction. Therefore, the assumption at the beginning is not true and we have  $Z_{\alpha} \notin G_{12}$.
\\
(b) Now suppose that  $E_1 \sim E_2$ without CM. In this setting, the isogeny group $\textup{Hom}(E_1,E_2)$ is generated by a nontrivial isogeny $\phi$, and the argument in (a) remains valid upon setting  $c_0 = c_1 = 0$ and $\beta_1 = 0$. This concludes the proof.
\end{proof}

Now we apply the above lemma to prove \Cref{ord_4torsion}. Let $\kappa_t^{\prime} \in \ch^2(E_1 \times E_2 \times k(E_t))$ be the restriction of the torsion class $[Z_t^{\prime}] = 2([t]-[e_1]) \times [\id] \in \ch^2(E_1 \times E_2 \times E_t)$. Recall from   \Cref{order_CycleInProduct} that $\ord([Z_t^{\prime}]) = \frac{n}{2}$ (note that $4 \mid n$). We need to show that $\ord(\kappa_t^{\prime}) = \frac{n}{2}$ as well.    
\begin{proof}[Proof of \Cref{ord_4torsion}]
	Let $4 \mid n$. Since $t \in E_1$ is a torsion point of order $n$, the torsion class $\alpha \coloneqq \frac{n}{2} \cdot ([t]-[e]) \in \ch^1(E_1)$ is of order two. Applying  \Cref{case2_lem}  to  $Z_{\alpha} \coloneqq  \alpha \times [\id] = \frac{n}{4} \cdot [Z_t^{\prime}]$ (note that $E_t \simeq E_2$), we see that $ \frac{n}{4} \cdot  [Z_t^{\prime}]$ is not from $\ch^1(E_1 \times E_2) \otimes \ch^1(E_t)$, 
	hence $\frac{n}{4} \cdot  \kappa_t^{\prime} \neq 0$ by \Cref{lem_bloch}.
	Also note that  $\frac{n}{2} \cdot  \kappa_t^{\prime} =0$ by construction. Therefore, by writing $n = 2^{s+2} \cdot k$, where $2 \nmid k$ and $s \geqslant 0$, we have $\ord(\kappa_t^{\prime}) = 2^{s+1} \cdot p$ for some positive integer $p \mid k$. 
	
	We argue by contradiction as in the proof of \Cref{order_kappa'_case1} to show that $p = k$. Assume that $p < k$. Then 
		the following pushfoward (see \Cref{ker_sum} for the notation) $$(\id_{E_1} \times \Sigma)_{\ast}((2^{s+1}\cdot p ) [Z_t^{\prime}]) = (2^{s+3}\cdot p)([t]-[e]) \times [E_2] \in \ch^1(E_1 \times E_2)$$ is a torsion class of order $\frac{k}{p}>1$, hence nontrivial. On the other hand, since $\ord(\kappa_t^{\prime}) = 2^{s+1} \cdot p$, again we have  $(2^{s+1} \cdot p )[Z_t^{\prime}] \in G_{12} \cap G^{2,1,1}_{\lhom} \subset G_{\lhom}^{1,1,2} \subset \ker (\textup{id}_{E_1} \times \Sigma)_{\ast} $ by  \Cref{chow_tripleprod} and  \Cref{ker_sum}, which is absurd. Therefore, we conclude that $p = k$ and $\ord(\kappa_t^{\prime}) = 2^{s+1} \cdot k = \frac{n}{2}$.
\end{proof}

Our main result  \Cref{ccc_Kummer_ord} follows immediately from \Cref{order_kappa'_case1} and \Cref{ord_4torsion}. 

\begin{rmk} 
	 Recall from \cite[Prop.~5.1]{Huybrechts2014} that for a polarized K3 surface $(X,L)$ and a fixed integer $n>0$, there exist at most finitely constant cycle curves $C \in |L|$ of order $n$. This finiteness result already provides evidence for our main result \Cref{ccc_Kummer_ord}. In fact, for any torsion point $t \in E_1$, the associated elliptic constant cycle curve $E_t$ represents the fibre class  of the natural elliptic fibration $p \colon \textup{Kum}(E_1 \times E_2) \ra E_1/\pm \simeq \mbp^1$, hence the finiteness result implies that there are only finitely many torsion points $t \in E_1$ such that $\ord(E_t) = 1$. In particular, if $t \in E_1$ is of prime order $p$ and  $p$ is large enough, then $\ord(E_t) = p$. 
\end{rmk}

\section{An example for the excluded case}\label{counterexample}

	In this section,  we focus on the case excluded  from the statement of our main result  \Cref{ccc_Kummer_ord}, namely when $E_1$ and $E_2$ are non-isomorphic but isogeneous elliptic curves with CM, and $4 \mid n$.  This case is notably subtle. We construct a series  of examples where, for a torsion point $t \in E_1$ of order $n = 4$, both $\ord(E_t) = 1$ and $\ord(E_t) = 2$ can occur,  depending on certain numerical conditions on $E_1$ and $E_2$.
	\begin{prop} \label{eg_excluded}
			Let $m>0,d<0$ be two integers. Consider $E_1 \coloneqq \mbc/\mbz m \oplus \mbz \sqrt{d}$ and  $E_2  \coloneqq \mbc/\mbz \oplus \mbz \sqrt{d}$. Then there exists a torsion point $t \in E_1$ of order $4$ such that 
			\begin{equation*}
				\ord(E_t) =
				\begin{cases}
					1,\quad & 2 \mid m \ \textup{and} \ 2 \nmid d \\
					2,\quad & \textup{otherwise}.
				\end{cases}
			\end{equation*}
	\end{prop}
	 \begin{proof}
	 Denote by $\phi_{\lambda}$ the isogeny given by the multiplication by $\lambda \in \mbc$ (when $\lambda$ preserves the lattice).
	 One can check that $E_1,E_2$ are non-isomorphic isogeneous elliptic curves with CM, and the isogeny group $\Hom(E_1,E_2) = \mbz \cdot \phi_{1} \oplus \mbz \phi_{\sqrt{d}}$.  Let $v_0,v_1 \in H_1(E_1,\mbz)$ and $w_0,w_1 \in H_1(E_2,\mbz)$ denote  the generators $m,\sqrt{d}$ and $1,\sqrt{d}$ of the two lattices, repectively. Then the matrices of the two generators $\phi_1,\phi_{\sqrt{d}}$ under the above basis are
	\begin{gather}
		\phi_{1}\begin{bmatrix} v_0  \\ v_1 \end{bmatrix}
		=
		\begin{bmatrix}
			m &
			0 \\
			0 &
			1 
		\end{bmatrix} \cdot \begin{bmatrix}
			w_0  \\
			w_1 
		\end{bmatrix}, \ \ \phi_{\sqrt{d}}\begin{bmatrix} v_0  \\ v_1 \end{bmatrix}
		= \begin{bmatrix}
			0 &
			m \\
			d &
			0 
		\end{bmatrix} \cdot \begin{bmatrix}
			w_0  \\
			w_1 
		\end{bmatrix},
	\end{gather}
	and the induced one-cycles $[\phi_1], [\phi_{\sqrt{d}}] \in \ch^1(E_1 \times E_2)$ can be represented as 
	\begin{equation} \label{rep_isogeny}
		[\phi_{1}] = m v_1 \otimes w_0 - v_0 \otimes w_1, \ \ [\phi_{\sqrt{d}}] = - m v_1 \otimes w_1 + dv_0 \otimes w_0
	\end{equation}
	as topological two-cycles in $H_2(E_1 \times E_2, \mbz)$. 
	
	Now  take a torsion point $t \in E_1$ of order $n = 4$ such that $\Phi_{E_1}(2([t]-[e_1])) = \frac{1}{2}\int_{v_0} $ in  $J(E)$. Note that $t$ is not unique and indeed there are four choices.	By  \Cref{order_kappa_same}, we have $\ord(E_t) = \ord(\kappa_t^{\prime}) \mid d(n) = 2$. It remains to determine when $\ord(\kappa_t^{\prime}) = 1$. We repeat the  Abel--Jacobi map argument from the proof of \Cref{case2_lem}. 
	Indeed, $\ord(\kappa_t^{\prime}) = 1$ if and only if its compactification $[Z_t^{\prime}] = 2([t]-[e_1]) \otimes [\id]  \in \ch^2(E_1 \times E_2 \times E_t)$ is represented by a class from $\Hom(J(E_1) \times J(E_2)) \otimes \ch^1(E_t)_{\lhom}$,
	 i.e.\ there exist $\beta_1, \beta_2 \in \ch^1(E_{t})_{\lhom}$ such that
	 \begin{equation} \label{eq_withCM}
		[Z_t^\prime] = 2([t]-[e_1]) \otimes [\id] = [\phi_1] \otimes \beta_1 + [\phi_{\sqrt{d}}] \otimes \beta_2 \in \ch^2(E_1 \times E_2 \times E_t).
	 \end{equation}
 	Since $[Z_t^{\prime}]$ is a nonzero two-torsion class (see \Cref{order_CycleInProduct}),  by the injectivity of the Abel--Jacobi map $\Phi$ on codimensional two torsion classes (\cite{CTSS83AJcodimtwo}), it is equivalent to require that there exist $\beta_1,\beta_2 \in \ch^1(E_t)$ such that
 		$$\Phi\big(2([t]-[e_1]) \otimes [\id]\big) = \Phi\big([\phi_1] \otimes \beta_1 + [\phi_{\sqrt{d}}] \otimes \beta_2 \big)\in J^3(E_1 \times E_2 \times E_t).
 	$$ Note that any $\beta_1,\beta_2$ satisfying the above equation  have to be two-torsion and are not both trivial. Since the Abel--Jacobi map $\Phi_{E_t}$ identifies the two-torsion subgroup $\ch^1(E_t)[2]$ with  $J(E)[2] = \{\frac{1}{2}\int_{\gamma} \mid \gamma \in H_1(E_t,\mbz)\}$, we see that  $\ord(E_t) = 1$ if and only if there exist $\gamma_1,\gamma_2$ in $H_1(E_t,\mbz)$ such that
    \begin{equation}\label{eq_countereg_int}
    	\frac{1}{2} \int_{v_0 \otimes [\id]} = \frac{1}{2} \int_{[\phi_1] \times \gamma_1} + \frac{1}{2} \int_{[\phi_{\sqrt{d}}] \times \gamma_2} \in J^3(E_1 \times E_2 \times E_t).
    \end{equation} 
Using the chosen basis of $H_1(E_1,\mbz)$,  $H_1(E_2,\mbz) \simeq H_1(E_t,\mbz)$ and plugging in \eqref{rep_isogeny}, we reformulate \eqref{eq_countereg_int} as: There exists $c_0,c_1,d_0,d_1 \in \mbz$ such that
 	\begin{align*}
 		v_0 \otimes (-w_0 \otimes w_1 + w_1 \otimes w_0) \equiv &(m v_1 \otimes w_0 - v_0 \otimes w_1) \otimes (c_0w_0+c_1w_1) \\
 		& + \big(-m v_1 \otimes w_1 + dv_0 \otimes w_0\big) \otimes (d_0w_0 + d_1w_1) \\
 		&\qquad \qquad \qquad \qquad \qquad \textup{mod} \ \ 2H_3(E_1 \times E_2 \times E_t,\mbz).
 	\end{align*} 
 
 	To sum up, under our choice of the torsion point $t \in E_1$ of order $n = 4$, we have $\ord(E_t) \mid 2$, and $\ord(E_t) = 1$ if and only if the following system of congruence equations 
 	\begin{equation*}
 		1 \equiv dd_1 \equiv c_0, \ \   mc_0 \equiv c_1 \equiv 0, \ \ 
 		0 \equiv mc_1 \equiv md_0    , \ \ md_1 \equiv dd_0 \equiv 0 \pmod 2
 	\end{equation*} 
   has a solution $(c_0,c_1,d_0,d_1)$, namely if and only if $2 \mid m$ and $2 \nmid d$. In this case the solution is $c_1 \equiv d_0  \equiv  0$ and $c_0 \equiv d_1 \equiv 1 \pmod 2$. 
 	\end{proof}

\medskip

\printbibliography
\end{document}